\documentclass[12pt]{amsart}

\usepackage{amscd}
\usepackage{amsmath}
\usepackage{amssymb}
\usepackage{amsthm}
\usepackage{epsf}
\usepackage{latexsym}
\usepackage{verbatim}
\usepackage[dvips]{graphicx}
\usepackage[all, cmtip]{xy}
\usepackage{epstopdf}

\newtheorem{definition}{Definition}
\newtheorem{theorem}{Theorem}[section]
\newtheorem{lemma}[theorem]{Lemma}

\newtheorem{corollary}[theorem]{Corollary}
\newtheorem{proposition}[theorem]{Proposition}
\newtheorem{varexample}[theorem]{Example}

\newcommand{\PP}{\mathbb{P}\,}
\newcommand{\LL}{\mathcal{L}\,}
\newcommand{\OO}{\mathcal{O}\,}
\newcommand{\II}{\mathcal{I}\,}
\newcommand{\Z}{\mathbb{Z}\,}

\begin{document}
\title{Rational Fibrations of $\overline{M}_{5,1}$ and $\overline{M}_{6,1}$}
\author{David Jensen}
\email{djensen@math.sunysb.edu}
\address{Stony Brook University\\Department of Mathematics\\Stony Brook, NY 11794-3651}
\date{}
\bibliographystyle{alpha}
\begin{abstract}
We construct rational maps from $\overline{M}_{5,1}$ and $\overline{M}_{6,1}$ to lower-dimensional moduli spaces.  As a consequence, we identify geometric divisors that generate extremal rays of the effective cones for these spaces.
\end{abstract}
\maketitle

\section{Introduction}

The moduli spaces of curves $\overline{M}_{g,n}$ are among the most important objects in algebraic geometry.  One natural approach to studying these spaces could be to construct morphisms from these moduli spaces to other, simpler spaces.  There is, however, a problem with this approach -- the moduli spaces of curves do not admit any non-trivial morphisms to lower dimensional varieties.  In particular, it has been shown in \cite{GKM} that every fibration (with connected fibers) of $\overline{M}_{g,1}$ factors through the map $\overline{M}_{g,1} \to \overline{M}_g$ given by forgetting the marked point.  We therefore cannot expect to learn much about the geometry of $\overline{M}_{g,1}$ if we restrict our attention to such morphisms.  We obtain a much richer picture, however, if we consider more generally rational maps.

Many questions about the rational maps of a variety $X$ can be restated in terms of the cone of effective divisors $\overline{NE}^1 (X)$.  For any divisor $D$, we may define the section ring
$$ R(X,D) = \bigoplus_{n \in \mathbb{N}} H^0 (X, \mathcal{O}_X (D)^{\otimes n} ). $$
If $D$ is effective and $R(X,D)$ is finitely generated, then there is an induced rational map
$$f_D : X \dashrightarrow Proj(R(X,D)).$$
In many cases, $\overline{NE}^1 (X)$ provides us with a nice combinatorial parameterization of the variety's rational contractions.

In this paper, we exhibit rational contractions from $\overline{M}_{5,1}$ and $\overline{M}_{6,1}$ to moduli spaces of pointed rational curves.  As a consequence, we identify extremal rays of the corresponding effective cones.  We note that this differs from the standard approach to finding extremal rays, which is to construct \textit{birational} contractions.

The study of $\overline{NE}^1 ( \overline{M}_g )$ was pioneered by Eisenbud, Harris and Mumford in their proof that $\overline{M}_g$ is of general type for genus $g \geq 24$ \cite{HMum} \cite{HE}.  A key step in their proof is the computation of the class of certain geometric divisors on $\overline{M}_g$.  In particular, their argument makes use of the Brill-Noether divisors, which are defined whenever $g+1$ is composite.  For the remaining genera, they complete the proof using a different set of divisors, known as the Gieseker-Petri divisors.  Part of the motivation for studying these divisors was the observation that, for small values of $g$, these divisors play a special role -- for $g \leq 11$, $g \neq 10$, there is an extremal ray of $\overline{NE}^1 ( \overline{M}_g )$ generated by either a Brill-Noether or Gieseker-Petri divisor (see Remark 2.17 in \cite{FarkasGG}).

In \cite{Logan}, Logan introduced the notion of \textbf{pointed Brill-Noether divisors}.  In Theorem \ref{Genus5Theorem}, we show that one of these divisors generates an extremal ray of $\overline{NE}^1 ( \overline{M}_{5,1} )$.

\begin{definition}
Let $Z = (a_0 , \ldots , a_r )$ be an increasing sequence of nonnegative integers with $\alpha = \sum_{i=0}^r a_i - i$.  Let $BN^r_{d,Z}$ be the closure of the locus of pointed curves $(C,p) \in M_{g,1}$ possessing a $g^r_d$ on $C$ with ramification sequence $Z$ at $p$.  When $g+1 = (r+1)(g-d+r)+ \alpha$, this is a divisor in $\overline{M}_{g,1}$, called a \textbf{pointed Brill-Noether divisor}.
\end{definition}

Our first result is:

\begin{theorem}
\label{Genus5Theorem}
The Weierstrass divisor $BN^1_{5,(0,5)}$ generates an extremal ray of $\overline{NE}^1 ( \overline{M}_{5,1} )$.
\end{theorem}

On $\overline{M}_{6,1}$, we define a divisor not of pointed Brill-Noether type, the divisor of ``nodes of $g^2_6$'s''.

\begin{definition}
Let $D_6$ be the closure of the locus of pointed curves $(C,p) \in M_{6,1}$ possessing a $g^2_6$ $\mathcal{L}$ and a point $p' \in C$ such that
$$h^0 (C, \mathcal{L} - p - p' ) \geq 2.$$
\end{definition}

As far as we are aware, the divisor $D_6$ does not appear earlier in the literature.  Its numerical class has recently been computed by Nicola Tarasca in \cite{Tarasca}.  We prove the following:

\begin{theorem}
$D_6$ generates an extremal ray of $\overline{NE}^1 ( \overline{M}_{6,1} )$.
\end{theorem}

Our strategy is to construct geometrically meaningful rational maps from $\overline{M}_{g,1}$ to other moduli spaces.  Our results follow by examining the images of divisors under these maps.  We describe the maps here.

The general genus 6 curve admits an embedding into a smooth quintic del Pezzo surface $Y$ as a section of $\vert -2K_Y \vert$.  This embedding is unique up to an automorphism of the surface.  By forgetting the curve and simply remembering the marked point, we obtain a rational map
$$ \phi_6 : \overline{M}_{6,1} \dashrightarrow Y/S_5 \cong \widetilde{M}_{0,5} .$$

The general genus 5 curve $C$ admits a canonical embedding into $\PP^4$ as the complete intersection of 3 quadrics.  For any point $p \in C$, the set of quadrics containing both $C$ and the tangent line to $C$ at $p$ forms a 2-dimensional vector space.  Let $Z$ be the intersection of all the quadrics in this space, which is a degree 4 del Pezzo surface.  Now, let $H \subseteq \PP^4$ be the osculating hyperplane to $C$ at $p$, and consider the intersection $H \cap Z$.  Since $T_p C \subset H \cap Z$, we may write $H \cap Z$ as the union of two components $T_p C \cup R$, where $R$ is generically a twisted cubic in $H$.

Notice that, since $C$ is a complete intersection of quadrics, it has no trisecants, and thus the intersection multiplicity of $C$ and $T_p C$ at $p$ on $Z$ must be 2.  Since $H$ intersects $C$ at $p$ with order of vanishing 4 or more, we see that $R$ must be tangent to $C$.  The three curves $C$, $R$, and $T_p C$ all therefore have the same tangent direction at $p$.  If we blow up $Z$ at $p$, the strict transforms of all three curves will pass through the same point on the exceptional divisor $E$.  If we then blow up again at this point, the new exceptional divisor will be a $\PP^1$ with 4 marked points on it -- namely, the points of intersection of this $\PP^1$ with the strict transforms of $C$, $R$, $E$, and $T_p C$.  In this way, we obtain a rational map
$$ \phi_5 : \overline{M}_{5,1} \dashrightarrow \overline{M}_{0,4} \text{       (see Figure \ref{Genus 5 Map}).}$$

\begin{figure}[!htb]
\includegraphics[scale=.9]{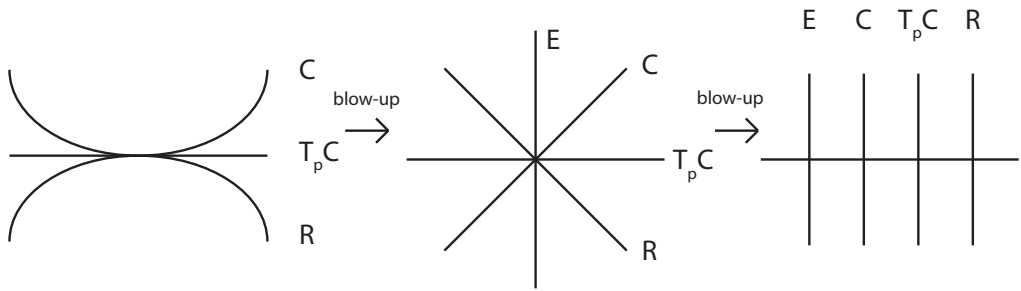}
\caption{The map $\phi_5$}
\label{Genus 5 Map}
\end{figure}

It is natural to ask to what extent our techniques generalize.  In an earlier paper we identify extremal rays of $\overline{M}_{g,1}$ generated by pointed Brill-Noether divisors when $g=3,4$ \cite{Genus3and4}.  In each case the maps we construct are specific to the genus, but pointed Brill-Noether divisors exist in every genus and it is possible that there is always a corresponding extremal ray.  There are higher-genus analogues of the divisor $D_6$ as well, as discussed in \cite{Tarasca}.  More generally, we expect that Proposition \ref{Negativity of Fibration} may have several applications, as it gives a new method for identifying extremal rays.

The outline of the paper is as follows.  In Section 2 we develop tools for studying the birational geometry of a variety by considering a certain type of rational fibration.  In Sections 3 and 4, we use these tools to obtain our results on $\overline{M}_{6,1}$ and $\overline{M}_{5,1}$, respectively.  While the results in Section 3 follow directly, in Section 4 we must show that the map $\phi_5$ decomposes as the composition of a birational contraction and a morphism.  In order to show this, we use techniques from geometric invariant theory.

Notation:  Following \cite{Rulla}, we refer to any element of $\overline{NE}^1 (X)$ as an effective divisor.  To avoid confusion, we use the term ``codimension one subvariety'' where applicable.

\textbf{Acknowledgements:}  This work was prepared as part of my doctoral dissertation under the direction of Sean Keel.  I am grateful to him for his numerous suggestions and ideas.  I would also like to thank Gavril Farkas, Brendan Hassett, David Helm and Eric Katz for several helpful conversations.

\section{Rational Fibrations}

In this section, we develop tools for studying effective cones.  One well-known such technique involves the construction of birational maps (see, for example, \cite{Rulla}).  This is because the exceptional divisors of a birational contraction $X \dashrightarrow Y$ generate extremal rays of $\overline{NE}^1 (X)$.  This idea has been exploited by many authors in the study of various moduli spaces, including $\overline{M}_g$ \cite{Rulla}, $\overline{M}_{0,n}$ \cite{CT}, and the Kontsevich space of stable maps \cite{CHS}.  In our case, however, the maps we construct are not birational -- they in fact have very low-dimensional images.  In what follows, we will see that extremal rays can be similarly obtained from these rational fibrations.  We begin with a few definitions, which can all be found in \cite{HK}.

\begin{definition}
Let $f: X \dashrightarrow Y$ be a rational map between normal projective varieties.  Let $(p,q) : W \to X \times Y$ be a resolution of $f$ with $W$ projective and $p$ birational.  We call $f$ a \textbf{birational contraction} if it is birational and every $p$-exceptional divisor is $q$-exceptional.  More generally, we say that $f$ is a \textbf{contraction} if every $p$-exceptional divisor is $q$-fixed.  For a $\mathbb{Q}$-Cartier divisor $D \subset Y$, we define $f^* (D)$ to be $p_* (q^* (D))$.
\end{definition}

Examples of contractions include morphisms, birational contractions, and compositions of these.  Rational contractions are exactly the maps corresponding to complete linear series.  Such maps are often useful for identifying extremal rays.  One example is the statement about birational maps alluded to above.  This fact is a direct consequence of the following, known variously as the ``Kodaira Lemma'' or ``Negativity of Contraction''.

\begin{lemma}  (See \cite{Kollar}, \cite{Rulla})
\label{Negativity of Contraction}
Let $f: X \to Y$ be a proper birational morphism, with $X$ regular in codimension one.  Let $\{ E_i \}_{i=1}^n$ be a collection of $f$-exceptional divisors, $M$ an $f$-nef $\mathbb{Q}$-Cartier divisor on $X$, and $P \in \overline{NE}^1 (X)$ any element such that there is a sequence of $\mathbb{Q}$-Cartier divisors $\{ D_j \} \to P$ such that no $E_i$ is contained in the support of any $D_j$.  If $L$ is a $\mathbb{Q}$-Cartier divisor on $Y$ such that
$$ f^* L = M + P + \sum_{i=1}^n e_i E_i ,$$
then $e_i \geq 0$ $\forall i$.
\end{lemma}

One technical issue that arises when studying effective cones is that not every element of $\overline{NE}^1 (X)$ corresponds to a codimension one subvariety of $X$.  Instead, an element of this cone is a \textit{limit} of codimension one subvarieties.  A major tool used in this work for dealing with this issue is the following:

\begin{lemma}  (See \cite{Rulla})
\label{Limits of Divisors}
Let $X$ be a projective variety and suppose that $\sum_{i=1}^n \mathbb{R}^{\geq 0} E_i$ is a subcone of $\overline{NE}^1 (X)$ generated by codimension one subvarieties $E_i$.  Then for any $P \in \overline{NE}^1 (X)$, one can write
$$ P = Q + \sum_{i=1}^n c_i E_i $$
where $c_i \geq 0$ $\forall i$ and $Q$ is the limit of effective $\mathbb{Q}$-Cartier divisors $\{ D_j \}$ such that no $E_i$ is contained in the support of any $D_j$.
\end{lemma}

We now turn our attention to a specific type of rational contraction -- the composition of a birational contraction and a morphism.  Throughout the remainder of this section, we let $X, Y$ and $Z$ be normal projective varieties with dim$Y <$ dim $X$, $f: X \dashrightarrow Z$ a birational contraction, $c: Z \to Y$ a proper morphism, and $h: X \dashrightarrow Y$ the composition.  Let $W$ be a resolution of $f$, giving us the following diagram.
$$\xymatrix{
& & W \ar[lld]_{p} \ar[d]_{q} \ar[rdd]^e & \\
X \ar@{.>}[rr]^f \ar@{.>}[rrrd]^h & & Z \ar[rd]_c & \\
& & & Y }$$
Note that, if $X$ is a Mori Dream Space, then all rational contractions are of this type (see Proposition 1.11 in \cite{HK}).  In the later sections, we will see that both of the maps we construct are of this type as well.  In what follows, we describe how to use these fibrations to identify extremal rays.

\begin{proposition}
\label{Negativity of Fibration}
Let $D_1$ and $D_2$ be distinct irreducible codimension one subvarieties of $X$ such that $\overline{h( D_1 )} = \overline{h( D_2 )} \neq Y$.  Then $D_1$ generates an extremal ray of $\overline{NE}^1 (X)$.  Furthermore, the divisor $D_1$ does not move.
\end{proposition}

\begin{proof}

The result is well-known in the case that $D_1$ is contracted by $f$ (see, for example, \cite{Rulla}) .  In what follows we assume that this is not the case.

Let $\tilde{D_1}$, $\tilde{D_2}$ denote the strict transforms of $D_1$, $D_2$ in $W$.  Let $D \subset Y$ be an irreducible codimension one subvariety containing $e( \tilde{D_1} )$.  By assumption, there exist irreducible divisors $B_i$ in $W$ such that
$$ e^* D = a_1 \tilde{D_1} + a_2 \tilde{D_2} + \sum b_i B_i $$
where $a_i , b_i >0$ for all $i$.  Note that some of the divisors $B_i$ may be exceptional divisors of the map $p$.  We will denote these by $E_i$ and rewrite the above expression as
$$ e^* D = a_1 \tilde{D_1} + a_2 \tilde{D_2} + \sum b_i B_i + \sum e_i E_i .$$

Now, suppose that $P_1 ,P_2 \in \overline{NE}^1 (X)$ such that $P_1 + P_2 = D_1$.  Then
$$ p^* P_1 + p^* P_2 = p^* D_1 = \alpha \tilde{D_1} + \sum_{i=1}^n \beta_i E_i .$$
By Lemma \ref{Limits of Divisors}, we may write $$ p^* P_k = Q_k + c_k \tilde{D_1} + \sum_{i=1}^n d_{ik} E_i $$
where $c_k , d_{ik} \geq 0$ and $Q_k$ is a limit of effective divisors whose support does not contain $\tilde{D_1}$ or $E_i$.  Note that $Q_1 + Q_2 + (c_1 + c_2 - \alpha ) \tilde{D_1}$ is in the span of the $E_i$'s.  Hence, if $c_1 + c_2 - \alpha \geq 0$, then by extremality, $Q_1$ and $Q_2$ are in the span of the $E_i$'s as well.  It follows that $P_k = p_* p^* P_k$ is linearly equivalent to a multiple of $D_1$, and we are done.

If, on the other hand, $c_1 + c_2 - \alpha < 0$, then by solving for $\tilde{D_1}$ in the expression above and replacing $Q_1$, $Q_2$ with positive multiples, we have
$$ \tilde{D_1} = Q_1 + Q_2 + \sum_{i=1}^n \gamma_i E_i .$$
Substituting this into the expression above, we obtain
$$ e^* D = a_1 (Q_1 + Q_2 ) + a_2 \tilde{D_2} + \sum b_i B_i + \sum d_i E_i $$
where $a_i , b_i > 0$.  Note that, since $e^* D = q^* c^* D$ and all of the $p$-exceptional divisors are $q$-exceptional, by Lemma \ref{Negativity of Contraction} we have $d_i \geq 0$ for all $i$ as well.

Now, let $F$ be a general fiber of the map $\tilde{D_1} \to e( \tilde{D_1} )$.  Furthermore, let $A$ be a general ample divisor in $F$, $n = dimF -1$, and $C = A^n$.  Notice that, since $C$ covers $\tilde{D_1}$, its intersection with any divisor whose support does not contain $\tilde{D_1}$ is nonnegative.  In particular, it has nonnegative intersection with $Q_i$, $B_i$, and $E_i$.  Moreover, since $C$ is contained in a fiber, $C \cdot e^* D = 0$.  Indeed, it is clear that $C$ does not intersect the pullback of a general ample divisor.  Hence, by linearity of the intersection pairing, it follows that $C \cdot e^* D = 0$ for every divisor $D$ in $Y$.  Finally, since $\tilde{D_2}$ surjects onto $e( \tilde{D_1} )$, we know that $C \cdot \tilde{D_2} > 0$, a contradiction.  It follows that $D_1$ generates an extremal ray of $\overline{NE}^1 (X)$.

To see that $D_1$ does not move, consider the case where $P_2 = 0$ and some multiple $mP_1$ is an irreducible divisor different from $D_1$.  We write $\tilde{P_1}$ for the strict transform of this divisor.  We then have
$$ mp^* D_1 =  m \alpha \tilde{D_1} + \sum_{i=1}^n m \beta_i E_i = \alpha ' \tilde{P_1} + \sum_{i=1}^n \beta_i ' E_i .$$
By replacing $\tilde{P_1}$ with a positive multiple, we may substitute to obtain:
$$ e^* D = a_1 \tilde{P_1} + a_2 \tilde{D_2} + \sum b_i B_i + \sum d_i E_i $$
where, for the same reason as above, all of the coefficients are nonnegative.  We then obtain a contradiction in the same manner as before.

\end{proof}

\section{Genus 6 Curves}

In order to establish our results about the map $\phi_6$, we must show that it arises as the composition of a birational contraction with a morphism.  Recall that the general genus 6 curve admits an embedding into a smooth quintic del Pezzo surface $Y$ as a section of $\vert -2K_Y \vert$, and that this embedding is unique up to an automorphism of the surface.  If we let
$$Z = \{ (C,p) \in \PP (H^0 (Y, -2K_Y )) \times Y \vert p \in C \},$$
there is then a rational map
$$\xymatrix{
\overline{M}_{6,1} \ar@{.>}[r]^f \ar@{.>}[rd]^{\phi_6} & Z/S_5 \ar[d]^{c} \\
 & Y/S_5 \ar@{=}[r] & \widetilde{M}_{0,5}. }$$

\begin{proposition}
The map $f$ is a birational contraction.
\end{proposition}

\begin{proof}

For the first part, it suffices to exhibit a morphism $f^{-1}: U \to \overline{M}_{6,1}$, where $U \subset Z/S_5$ is open with complement of codimension $\geq 2$ and $f^{-1}$ is an isomorphism onto its image.  To see this, let $U \subset Z/S_5$ be the set of all moduli stable pointed curves $(C,p) \in Z/S_5$.  Notice that the complement of $U$ is strictly contained in the locus of singular curves, which is an irreducible hypersurface in $Z/S_5$.  It follows that the complement of $U$ has codimension $\geq 2$.

By the universal property of the moduli space, since
$$U \to \PP H^0 (Y, -2K_Y )/S_5$$
is a family of moduli stable curves, it admits a unique map $U \to \overline{M}_{6,1}$.  Since the embedding of a genus 6 curve in $Y$ is unique up to an automorphism of $Y$, two such curves are isomorphic if and only if they differ by an element of $S_5$.  Since the general genus 6 curve possesses no non-trivial automorphisms, it follows that this map is generically injective.

\end{proof}

We now consider the images of divisors in $\overline{M}_{6,1}$ under the map constructed above.  One geometric divisor of interest is the Gieseker-Petri divisor $GP_6$.  As mentioned in the introduction, the Gieseker-Petri divisors played a key role in Harris and Eisenbud's proof the $\overline{M}_g$ is of general type for $g$ sufficiently large \cite{HE}.  By definition, $GP_6$ is the closure of the locus of smooth curves $C$ admitting a $g^1_4$ $L$ such that the multiplication map
$$ H^0 (C,L) \otimes H^0 (C,K-L) \to H^0 (C,K) $$
fails to be injective.  While the typical genus 6 curve admits 5 $g^1_4$'s, the general element of $GP_6$ admits fewer than this expected number.

Another divisor of interest is the one denoted $D_6$ above.  Recall that $D_6$ is the closure of the locus of pointed curves $(C,p) \in M_{6,1}$ possessing a $g^2_6$ $L$ and a point $p' \in C$ such that
$$h^0 (C, L - p - p' ) \geq 2.$$
We will show that $D_6$ and $GP_6$ have the same image under the map $\phi_6$.

\begin{proposition}
\label{Genus 6 Result}
Let $\Delta \subset \widetilde{M}_{0,5}$ be the boundary divisor.  Then
$$\overline{\phi_6 ( GP_6 )} = \overline{\phi_6 ( D_6 )} = \Delta.$$
\end{proposition}

\begin{proof}

The general genus 6 curve has 5 $g^2_6$'s, corresponding to the 5 blow-downs $Y \to \PP^2$.  From this description, it is clear that a point on a genus 6 curve will map to a node under some $g^2_6$ if and only if it is contained in a $-1$ curve on $Y$.  It follows that $\overline{\phi_6 ( D_6 )} = \Delta$.

We now prove the statement about the Gieseker-Petri divisor.  Let $\mathcal{C} \to B$ be a family of curves over a DVR such that the general fiber is a general genus 6 curve and the special fiber is a general element of $GP_6$.  By definition, there exist line bundles $L_1 , L_2$ on $\mathcal{C}$ such that the restriction of $L_1$ and $L_2$ to the general fiber are distinct $g^1_4$'s, whereas their restrictions to the special fiber are identical $g^1_4$'s.  The linear series $\vert L_1 \vert$ and $\vert L_2 \vert$ determine a map from $\mathcal{C}$ to $\PP^1 \times \PP^1$, sending the general fiber birationally onto a curve with 3 nodes, and the special fiber to 4 times the diagonal.  Blowing up at the three nodes, we obtain a map from $\mathcal{C}$ to $Y$.  Note that the image of the special fiber under this map is a union of $-1$ curves.  In particular, it maps to the sum of 4 times a $-1$ curve, plus 2 times each of the 3 $-1$ curves meeting it.

$$\xymatrix{
\mathcal{C} \ar[rr] \ar[rrd]^{\vert L_1 \vert \times \vert L_2 \vert} \ar[d] & & Y \ar[d] \\
B & & \PP^1 \times \PP^1 }$$

Because this family is sufficiently general and $\widetilde{M}_{0,5}$ is projective, it follows that $\overline{ \phi_6 ( GP_6 )} = \Delta$.

\end{proof}

\begin{theorem}
The divisor $D_6$ generates an extremal ray of $\overline{NE}^1 ( \overline{M}_{6,1} )$.
\end{theorem}

\begin{proof}
This follows immediately from Propositions \ref{Negativity of Fibration} and \ref{Genus 6 Result}.
\end{proof}

\section{Genus 5 Curves}

In order to obtain our results on $\overline{M}_{5,1}$, we make a similar argument to that of section 3.  Specifically, we first exhibit the map $\phi_5$ as the composition of a birational contraction and a morphism, and then apply Proposition \ref{Negativity of Fibration} to show that certain effective divisors generate extremal rays of  $\overline{NE}^1 ( \overline{M}_{5,1} )$.  To see that this map does indeed admit the desired decomposition, we use techniques from geometric invariant theory.  For a more detailed discussion of variation of GIT, see \cite{DH} and \cite{Thaddeus}.

Throughout, we will make frequent use of Mumford's numerical criteria.  Given a reductive group $G$ acting on a variety $X$, and a one-parameter subgroup $\lambda : \mathbb{C}^* \to G$, we choose coordinates so that $\lambda$ is diagonal.  In other words, it is given by $diag(t^{a_1}, t^{a_2},  \ldots , t^{a_n} )$.  We will refer to the $a_i$'s as the weights of the $\mathbb{C}^*$ action.  For a point $x \in X$, Mumford defines
$$\mu_{\lambda} (x) = min(a_i \vert x_i \neq 0).$$
Then $x$ is stable (semistable) if and only if $\mu_{\lambda} (x) < 0$ (resp. $\mu_{\lambda} (x) \leq 0$) for every nontrivial 1-parameter subgroup $\lambda$ of $G$ \cite{Mumford}.

Now, let $Y = \PP^4$ and $Z = G(3, \OO_Y (2))$ be the Grassmmannian of 3-dimensional subspaces of the space of quadrics in $Y$.  Since a general genus 5 canonical curve is the complete intersection of 3 quadrics in $\PP^4$, the general point in $Z$ corresponds to a genus 5 curve.  For a given such canonical curve $C$, we will write $\II_C (2)$ for the vector space of quadrics containing the $C$.  Now, let
$$ X = \{ (\II_C (2) ,p) \in Z \times Y \vert p \in Q \text{   } \forall Q \in \II_C (2) \}. $$
We denote the various maps as in the following diagram.
$$\xymatrix{
X \ar[d]^f \ar[r]^i & Z \times Y \ar[d]^{\pi_2} \ar[r]^{\pi_1} & Y \\
Z \ar[r]^{id} &Z}$$
Since $X$ is a Grassmmannian bundle over $\PP^4$, it is smooth, and $PicX \cong \Z \times \Z$.  We will write $\OO_X (a,b)$ to denote $(i \circ\pi_1 )^*(\OO_{\PP^4}(a)) \otimes f^*(\OO_Z (b))$.  There is a natural action of $Aut(Y) = PSL(5, \mathbb{C} )$ on $X$.  Our goal is to study the GIT quotients of $X$ by the action of this group.

Notice that, since the general point of $X$ corresponds to a pointed genus 5 curve, there is a rational map $X \dashrightarrow \overline{M}_{5,1}$.  Our first step is to calculate the pullback of the pointed Brill-Noether divisors under this map.

\begin{proposition}
The pullback of every pointed Brill-Noether divisor under the map $X \dashrightarrow \overline{M}_{5,1}$ is numerically equivalent to $\OO_X (3,2)^{\otimes n}$ for some $n \geq 0$.

\end{proposition}

\begin{proof}

To prove this, we introduce two test curves on $X$.  Let $F_1$ be a fiber of a general point in $Z$ under the map $f: X \to Z$.  In other words, $F_1$ is obtained by fixing a general genus 5 curve $C$ and varying the marked point.  Notice that $\OO_X (0,1)$ has intersection number 0 with $F_1$, and $\OO_X (1,0)$ has intersection number $deg(C) = 8$ with $F_1$.

Let $S$ be the complete intersection of two general quadric hypersurfaces in $Y$.  Then $S$ is a smooth del Pezzo surface of degree 4.  Fix a point $p \in S$ and let $F_2$ be a Lefschetz pencil of curves in $\vert -2K_S \vert$ through $p$ with marked point $p$.  Since $F_2$ lies entirely inside a fiber of the map $i \circ \pi_1 : X \to Y$, $\OO_X (1,0)$ has intersection number 0 with $F_2$.  Moreover, since the image of $F_2$ under the map $f: X \to Z$ is a line in $Z$, $\OO_X (0,1)$ has intersection number 1 with $F_2$.  It is clear that the class of a divisor on $X$ is determined uniquely by its intersection with $F_1$ and $F_2$.

Notice that the pullback of the (non-pointed) Brill-Noether divisor has intersection number 0 with both $F_1$ and $F_2$, and therefore pulls back to zero under this rational map.  Because every pointed Brill-Noether divisor is an effective combination of this divisor and the Weierstrass divisor, we therefore see that the pullbacks of all the pointed Brill-Noether divisors lie on a single ray in $NS(X)$.  It thus suffices to compute the intersection numbers of the pullback of a single such divisor with $F_1$ and $F_2$.  Here we examine the Weierstrass divisor, $W = BN^1_{5,(0,5)}$.

The general genus 5 curve possesses $4 \cdot 5 \cdot 6 = 120$ Weierstrass points, so $W \cdot F_1 = 120$.  To compute $W \cdot F_2$, we use the class of the Weierstrass divisor in $\overline{M}_{5,1}$ (see \cite{Cuk1}):
$$ W = 15 \omega - \lambda + 10 \delta_1 + 6 \delta_2 + 3 \delta_3 + \delta_4 . $$
Since $F_2$ is a Lefschetz pencil, $\delta_i = 0$ $\forall i > 0$.  The total space of this pencil is the blow-up of $S$ at 16 points, with $p$ being one of these points.  It follows that the Euler characteristic of the total space $F_2^{tot}$ is 24.  If $F^{gen}$ is a generic fiber of this pencil, then
$$ \chi ( F_2^{tot} ) = \chi ( \PP^1 )\chi ( F_2^{gen} ) + \# \{ \text{singular fibers} \} . $$
It follows that $\delta_0 = 24 + 2 \cdot 8 = 40$.  Moreover, $K_{F_2^{tot}}^2 = -12$, and so $\kappa = -12 + 2 \cdot 2 \cdot 8 = 20$.  Since $\lambda = 12( \kappa + \delta )$, we have $\lambda = 5$.  Finally, $\omega$ is the negative self-intersection of the section corresponding to the fixed point $p$, which is just the exceptional divisor lying over $p$ in $F_2^{tot}$.  It follows that $\omega = 1$, and so $W \cdot F_2 = 15 - 5 = 10$.

Together, these intersection numbers show that the pullback of $W$ is numerically equivalent to $\OO_X (15,10)$.  This concludes the proof.

\end{proof}

Our next result shows that this ray generates an edge of the $G$-effective cone.

\begin{proposition}

$\OO_X (3,2)$ lies on a boundary of $C^G (X)$.

\end{proposition}

\begin{proof}

Let $\LL = \OO_X (3,2)$.  It suffices to show that $X^{ss}( \LL ) \neq X^s ( \LL ) = \emptyset$.  It is clear that $X^{ss}( \LL ) \neq \emptyset$, since the pullbacks of all of the pointed Brill-Noether divisors are $G$-invariant sections of $\LL^{\otimes n}$ for some $n$.

To show that $X^s ( \LL ) = \emptyset$, we invoke the numerical criterion.  Let $(C,p) \in X$.  By change of coordinates, we may assume that $p = (0,0,0,0,1)$.  Furthermore, we may assume that the tangent line to $C$ at $p$ is the line $x_0 = x_1 = x_2 = 0$.  Note that the map
$$ \II_C (2) \to H^0 ( \PP^1 , \OO_{\PP^1} (2) - 2p ) $$
given by restricting the quadric to the tangent line $T_p C$ is a linear map.  Since the codomain has dimension 1, there are at least two linearly independent quadrics in $C$ containing $T_p C$.  Again, by change of coordinates, we may assume that the tangent spaces to these two quadrics at $p$ both contain the plane $x_0 = x_1 = 0$.  In addition, let $V \subseteq \II_C (2)$ be the subspace of quadrics containing $T_p C$.  Notice that, if $Q \in V$, then the restriction of $Q$ to this plane is the union of $T_p C$ and a second line through $p$.  Consider the map
$$ V \to H^0 ( \PP^2 , \OO_{\PP^1} (1) - p) $$
given by restricting this second line to $T_p C$.  As above, since the codomain has dimension 1, there must be a quadric in $V$ whose restriction to this plane is a double line.  By change of coordinates, we may assume that the tangent space to this quadric at $p$ is the hyperplane $x_0 = 0$.   So, if
$$ C = \sum_{0 \leq i_{\alpha} \leq j_{\alpha} \leq 4} (a_{i_0 ,j_0 , i_1 , j_1 , i_2 , j_2} ) x_{i_0} x_{j_0} \wedge x_{i_1} x_{j_1} \wedge x_{i_2} x_{j_2},$$
then $a_{i_0 ,j_0 , i_1 , j_1 , 4 , 4} = a_{i_0 ,j_0 , i_1 , j_1 , 3 , 4} = a_{i_0 ,j_0 , 3 , 3 , 2 , 4} = a_{1 , 4 , 2 , 3 , 2 , 4} = 0 $.  In particular, notice that $a_{i_0 ,j_0 , i_1 , j_1 , i_2 , j_2} = 0$ if
$$ i_0 + j_0 + i_1 + j_1 + i_2 + j_2 > 15. $$

Under the embedding determined by $\LL$, we write $(C,p)$ in terms of the basis of monomials of the form
$$ x_4^3 \prod_{\alpha = 1}^2 a_{i_{0 \alpha}, j_{0 \alpha}, i_{1 \alpha}, j_{1 \alpha}, i_{2 \alpha}, j_{2 \alpha}}. $$
Now, consider the 1-parameter subgroup with weights $(-2,-1,0,1,2)$.  It acts on the monomial above with weight $6 + 2(12 - (i_0 + j_0 + i_1 + j_1 + i_2 + j_2 )) $, which is negative when $i_0 + j_0 + i_1 + j_1 + i_2 + j_2 > 15$.  By assumption, this is not the case, so $(C,p) \notin X^s ( \LL )$.  Since $(C,p)$ was arbitrary, it follows that $X^s ( \LL ) = \emptyset$.

\end{proof}

We now let $\LL (0) = \OO_X (3,2)$,and $\LL_-$ be a line bundle lying in the chamber of $C^G (X)$ adjacent to $\LL (0)$.  By general variation of GIT, we know that there is a morphism $X^{ss} ( \LL_- )//G \to X^{ss} ( \LL (0))//G$.  Moreover, since the linearization $\LL (0)$ admits no stable points, the corresponding quotient is in some sense degenerate -- it satisfies the categorical definition of a quotient but not the geometric definition.  We show that $\phi_5$ is equal to the composition of the natural map $\overline{M}_{5,1} \dashrightarrow X^{ss} ( \LL_- )//G$ with this map.

\begin{proposition}
The map $X^{ss} ( \LL_- )//G \to X^{ss} ( \LL (0))//G$ is the same as the natural map $X^{ss} ( \LL_- )//G \to \overline{M}_{0,4}$ described in the introduction.
\end{proposition}

\begin{proof}

It suffices to show that, if two points in $X^{ss} ( \LL_-)$ have the same image under the map $X \dashrightarrow \overline{M}_{0,4}$, then they have the same image under the map $X^{ss} ( \LL_- )//G \to X^{ss} ( \LL (0))//G$.  To prove this, we show that both points lie in the same orbit closure.  As above, we assume that $p = (0,0,0,0,1)$, and that $C = Q_1 \cap Q_2 \cap Q_3$, where
$$ Q_{\alpha} = \sum_{0 \leq i \leq j \leq 4} a_{\alpha , i,j} x_i x_j. $$
Also, as above, we assume that $a_{\alpha ,i,j} = 0$ if $i+j > 3 + \alpha$.  Note that, by acting on this curve by a diagonal matrix, we may further assume that $a_{3,3,3} = - a_{3,2,4}$.  We also note that, since $(C,p) \in X^{ss} ( \LL_- )$, the terms $a_{1,0,4}, a_{2,1,4}, a_{3,2,4}$ are all nonzero.  This can be verified using one-parameter subgroups.  We can therefore scale all of the quadrics so that $a_{1,0,4} = a_{2,1,4} = a_{3,2,4} = 1$.

We now determine the image of $(C,p)$ under the map to $\overline{M}_{0,4}$.  Notice that the del Pezzo surface $S$ is the intersection $Q_1 \cap Q_2$, and the osculating hyperplane to $C$ at $p$ is cut out by $x_0 = 0$.  Taking the intersection of $S$ with this hyperplane and projecting onto the the tangent plane $x_0 = x_1 = 0$, $x_4 = 1$, we obtain a curve in $\mathbb{A}^2$ cut out by the following equation:
$$ x_2 ( a_{1,1,1}a_{1,2,2} x_2 - a_{1,1,1}a_{2,2,3}a_{1,1,3}x_3^2 + \text{higher order terms} ). $$
Simplifying this, we see two components.  $T_p C$ is cut out by $x_2 = 0$, and $R$ is cut out by
$$ a_{1,2,2} x_2 - a_{2,2,3}a_{1,1,3}x_3^2 + \text{higher order terms}. $$
We obtain the inverse image of $R$ in the blow-up by considering this equation along with $t x_2 = u x_3$ in $\mathbb{A}^2 \times \PP^1$.  If $t \neq 0$, we can set $t=1$ and substitute $x_2 = u x_3$ to obtain:
$$ a_{1,2,2} u x_3 - a_{2,2,3}a_{1,1,3} x_3^2 + \text{higher order terms}  $$
$$ = x_3 (a_{1,2,2} u - a_{2,2,3}a_{1,1,3} x_3 + \text{higher order terms} ), $$
which intersects the exceptional divisor at the point $u = 0$.  To blow up the resulting surface at this point, we repeat this process and substitute $u = u' x_3$ to obtain:
$$x_3 (a_{1,2,2} u' - a_{2,2,3}a_{1,1,3} + \text{higher order terms} ),$$
which intersects the exceptional divisor at the point $u' = \frac{a_{2,2,3}a_{1,1,3}}{a_{1,2,2}}$.  We therefore see that, in the coordinates $(t',u')$, the points of intersection of this $\PP^1$ with the strict transforms of $R$, $E$, and the tangent line to $C$ at $p$ are the points $( a_{2,2,3}a_{1,1,3}, a_{1,2,2} )$, $(0,1)$, and $(1,0)$ respectively.  A similar calculation shows that the strict transform of $C$ intersects the exceptional divisor at the point $(1,1)$.  We now show that every curve with the same ratio $\frac{a_{2,2,3}a_{1,1,3}}{a_{1,2,2}}$ lies in the same orbit closure.

Consider again the 1-parameter subgroup with weights $(-2,-1,0,1,2)$.  The flat limit of $(C,p)$ under this 1-parameter subgroup is cut out by the following quadrics:
$$ Q_1 = x_0 x_4 + a_{1,1,3}x_1 x_3 + a_{1,2,2}x_2^2 ,$$
$$ Q_2 = x_1 x_4 + a_{2,2,3}x_2 x_3 ,$$
$$ Q_3 = x_2 x_4 - x_3^2 .$$
Now, consider the image of these quadrics under the action of the diagonal matrix
$$diag(\alpha^{-1} , \alpha , \beta , 1 \beta^{-1} ).$$
The image is
$$ Q_1 = x_0 x_4 + \alpha^2 \beta a_{1,1,3} x_1 x_3 + \alpha \beta^3 a_{1,2,2}x_2^2 ,$$
$$ Q_2 = x_1 x_4 + \frac{\beta^2}{\alpha} a_{2,2,3}x_2 x_3 ,$$
$$ Q_3 = x_2 x_4 - x_3^2 .$$
Notice that, if $a_{1,1,3}$ and $a_{1,2,2}$ are nonzero, then you can choose values for $\alpha$ and $\beta$ that make $\alpha^2 \beta a_{1,1,3} = \alpha \beta^3 a_{1,2,2} = 1$.  This forces $\frac{\beta^2}{\alpha} a_{2,2,3} = \frac{a_{2,2,3}a_{1,1,3}}{a_{1,2,2}}$.  We therefore see that our curve lies in the same orbit closure as the one cut out by the three quadrics:
$$ Q_1 = x_0 x_4 + x_1 x_3 + x_2^2 ,$$
$$Q_2 = x_1 x_4 + \frac{a_{2,2,3}a_{1,1,3}}{a_{1,2,2}} x_2 x_3 ,$$
$$ Q_3 = x_2 x_4 - x_3^2 .$$
So all curves with the same such ratio are identified under the map $X^{ss} \to X^{ss} ( \LL )//G$.  (A similar argument shows this result to hold if either $a_{1,1,3} = 0$ or $a_{1,2,2} = 0$.)  This completes the proof.

\end{proof}

\begin{proposition}

There is a birational contraction
$$p : \overline{M}_{5,1} \dashrightarrow X^{ss}(-) // G.$$

\end{proposition}

\begin{proof}

It suffices to exhibit a morphism $p^{-1}: V \to \overline{M}_{5,1}$, where $V \subseteq X^{ss}(-)/G(-)$ is open with complement of codimension $\geq 2$ and $p^{-1}$ is an isomorphism onto its image.  Let $U \subseteq X^{ss}(-)$ be the set of all moduli stable pointed curves $(C,p) \in X^{ss}(-)$.  Notice that the complement of $U$ is strictly contained in the locus $\Delta$ of singular curves, which is an irreducible hypersurface in $X^{ss}(-)$.  Thus, in the quotient, we have the containment $(X^{ss}(-) \backslash U) // \subset \Delta // G \subset X^{ss}(-) // G$, and $\Delta // G$ is irreducible.  Notice, furthermore, that both $\Delta$ and $X^{ss}(-) \backslash U$ are $G$-invariant,  so if $(C,p) \in X^s(-) \backslash \Delta$ (respectively, $(C,p) \in X^s(-) \cap \Delta \cap U$), then the orbit of $(C,p)$ does not intersect $\Delta$ (respectively, $X^{ss}(-) \backslash U$).  Since this point is stable, this means that the image of $(C,p)$ is not contained in $\Delta // G$ (respectively, $(X^{ss}(-) \backslash U) // G$).  Thus, the containments $(X^{ss}(-) \backslash U) // G \subset \Delta // G$ and $\Delta // G \subset X^{ss}(-) // G$ are strict.  It follows that the complement of $U$ has codimension $\geq 2$.

By the universal property of the moduli space, since $U \to Z$ is a family of moduli stable curves, it admits a unique map $U \to Z \to \overline{M}_{5,1}$.  This map is certainly $G$-equivariant, so it factors uniquely through a map $U // G \to \overline{M}_{5,1}$.  Since every smooth complete intersection of 3 quadrics in $\PP^4$ is a canonical genus 5 curve, two such curves are isomorphic if and only if they differ by an automorphism of $\PP^4$.  It follows that this map is an isomorphism onto its image.

\end{proof}

This proposition places us in a setting where we may use the facts from birational geometry above.  In particular, we see that the map $\phi_5$ can be expressed as the composition of a birational contraction and a morphism.

We now consider the three boundary divisors on $\overline{M}_{0,4}$.

\begin{proposition}
\label{BoundaryDivisors}
The pointed Brill-Noether divisors $BN^1_{5,(0,5)}$, $BN^1_{4,(0,3)}$, and $BN^2_{6,(0,2,4)}$ are each contained in the pullback by $\phi_5$ of boundary divisors on $\overline{M}_{0,4}$.  (See Figures \ref{Weierstrass}, \ref{Other}, and \ref{Logan}.)
\end{proposition}

\begin{proof}
We first note that, by Riemann-Roch,  $BN^1_{5,(0,5)}$ is the locus of pointed curves $(C,p)$ such that $p$ is a Weierstrass point of $C$.  Now, if $(C,p)$ is an element of the Weierstrass divisor , then the osculating hyperplane to $C$ at $p$ vanishes to order 5 at $p$.  In terms of the intersection product on $S$, this means that $(C \cdot (T_p C + R))_p = 5$.  If $C$ is not trigonal, then no line intersects $C$ in three points, so $(C \cdot T_p C)_p = 2$.  This means that $(C \cdot R)_p = 3$.  In other words, the strict transforms of $C$ and $R$ pass through the same point of the exceptional divisor.  It follows that the Weierstrass divisor is contained in the pullback of the point pictured in Figure \ref{Weierstrass}.
\begin{figure}[!htb]
\begin{center}
\includegraphics[scale=2]{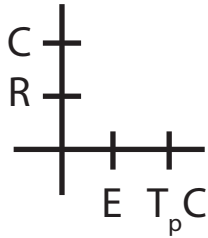}
\caption{Image of the Weierstrass divisor $BN^1_{5,(0,5)}$}
\label{Weierstrass}
\end{center}
\end{figure}

\begin{figure}[!htb]
\begin{center}
\includegraphics[scale=2]{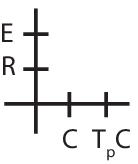}
\caption{Image of the Pointed Brill-Noether Divisor $BN^1_{4,(0,3)}$}
\label{Other}
\end{center}
\end{figure}

\begin{figure}[!htb]
\begin{center}
\includegraphics[scale=2]{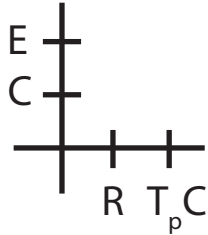}
\caption{Image of the Pointed Brill-Noether Divisor $BN^2_{6,(0,2,4)}$}
\label{Logan}
\end{center}
\end{figure}

Suppose that $2p + q + r$ is a $g^1_4$ on $C$, ramified at $p$.  By Riemmann-Roch, there is a plane in $\PP^4$ containing $q,r,$ and the line $T_p C$.  This means that the line through $q$ and $r$ intersects the line $T_p C$.  Since $S$ contains three points on this line, and $S$ is a complete intersection of quadrics, this line lies on $S$.  It follows that the 5 $g^1_4$'s on $C$ that are ramified at $p$ are cut out by divisors on $S$ of the form $T_p C + L$, where $L$ is a $-1$ curve on $S$ that intersects $T_p C$.  If $(C,p) \in BN^1_{4,(0,3)}$, then there must be a $-1$ curve $L$ on $S$, other than $T_p C$, that passes through $p$.  We see from this description that $R$ is the union of $L$ and another rational curve passing through $p$.  Since $R$ is singular at $p$, its inverse image under the first blow-up contains the exceptional divisor with multiplicity 2.  It follows that this pointed Brill-Noether divisor is contained in the pullback of the point pictured in Figure \ref{Other}.

Now, let $(C,p) \in BN^2_{6,(0,2,4)}$.  By definition, there exists a $g^2_6$ $D$ on $C$ with ramification sequence $(0,2,4)$ at $p$.  It follows that $D-2p$ and $K-D+2p$ are both $g^1_4$'s on $C$ that are ramified at $p$.  From the description above, we see that there are two $-1$ curves $L, L'$ on $S$ such that $L + L' + 2T_p C$ is a hyperplane section of $S$.  We therefore see that $R$ is the union of $L,L',$ and $T_p C$.  Since $R$ contains $T_p C$ as a component, this divisor is contained in the pullback of the point pictured in Figure \ref{Logan}.
\end{proof}

These descriptions of pointed Brill-Noether divisors determine fundamental properties of the map $\phi_5$.  In particular,

\begin{theorem}
\label{MainTheorem}
$BN^1_{4,(0,3)}$ is the pullback by $\phi_5$ of an ample divisor on $\overline{M}_{0,4}$.
\end{theorem}

\begin{proof}
By Proposition \ref{BoundaryDivisors}, we know that $BN^1_{4,(0,3)}$ is contained in the pullback by $\phi_5$ of a boundary divisor $\Delta$ on $\overline{M}_{0,4}$.  Hence $\phi_5^* \Delta = BN^1_{4,(0,3)} + E$, where $E$ is a sum of irreducible divisors on $\overline{M}_{5,1}$.  By definition, each of these irreducible divisors maps to the same point as $BN^1_{4,(0,3)}$, hence, by Proposition \ref{Negativity of Fibration}, either $E=0$ or $BN^1_{4,(0,3)}$ generates an extremal ray of $\overline{NE}^1 ( \overline{M}_{5,1} )$.

By Theorem 4.5 in \cite{Logan}, however, we know that $BN^1_{4,(0,3)} = BN^1_{5,(0,5)} + BN^1_3$, and thus $BN^1_{4,(0,3)}$ clearly does not lie on an extremal ray of $\overline{NE}^1 ( \overline{M}_{5,1} )$.  It follows that $BN^1_{4,(0,3)}$ is the pullback by $\phi_5$ of an ample divisor on $\overline{M}_{0,4}$.
\end{proof}

\begin{corollary}
The Weierstrass divisor $BN^1_{5,(0,5)}$ generates an extremal ray of $\overline{NE}^1 ( \overline{M}_{5,1} )$.
\end{corollary}

\begin{proof}
By Theorem \ref{MainTheorem}, we know that $BN^1_{5,(0,5)} + BN^1_3$ is the pullback by $\phi_5$ of a point on $\overline{M}_{0,4}$.  It follows that both irreducible divisors map to the same point, and hence, by Proposition \ref{Negativity of Fibration}, they both generate extremal rays of $\overline{NE}^1 ( \overline{M}_{5,1} )$.
\end{proof}

\begin{corollary}
$BN^2_{6,(0,2,4)}$ is numerically equivalent to a multiple of $BN^1_{4,(0,3)}$.
\end{corollary}

\begin{proof}
By \cite{Logan2}, we know that $BN^2_{6,(0,2,4)}$ is numerically equivalent to a linear combination of $BN^1_{5,(0,5)}$ and $BN^1_3$.  On the other hand, if it is not numerically equivalent to a multiple of $BN^1_{4,(0,3)}$, then by the above argument it generates an extremal ray.  This would imply that $BN^2_{6,(0,2,4)}$ is linearly equivalent to a multiple of either $BN^1_{5,(0,5)}$ or $BN^1_3$.  But Proposition \ref{Negativity of Fibration} shows that neither of these divisors move, so this is impossible.
\end{proof}

\bibliography{ref}

\end{document}